\documentclass{amsart}
\usepackage{a4wide}

\usepackage[all]{xy}

\newcommand{\U}{\mathcal U}
\newcommand{\V}{\mathcal V}
\newcommand{\cf}{\mathrm{cf}}

\newcommand\w{\omega}
\newcommand{\IN}{\mathbb N}

\newcommand{\IZ}{\mathbb Z}
\newcommand{\II}{\mathbb I}

\newcommand{\e}{\varepsilon}
\newcommand{\Ra}{\Rightarrow}
\newcommand{\Law}{\bar\Lambda}

\newtheorem{theorem}{Theorem}
\newtheorem{problem}{Problem}
\newtheorem{lemma}{Lemma}
\newtheorem{corollary}{Corollary}
\newtheorem{proposition}{Proposition}
\newtheorem{claim}{Claim}

\newtheorem{example}{Example}
\theoremstyle{definition}
\newtheorem{definition}{Definition}
\newtheorem{remark}{Remark}

\title{The Lawson number of a semitopological semilattice}
\author{Taras Banakh, Serhii Bardyla, Oleg Gutik}
\address{T.~Banakh: Ivan Franko National University of Lviv (Ukraine) and Jan Kochanowski University in Kielce (Poland)}
\email{t.o.banakh@gmail.com}
\address{S.~Bardyla: Institute of Mathematics, Kurt G\"{o}del Research Center, Vienna (Austria)}
\email{sbardyla@yahoo.com}
\address{O.~Gutik: Ivan Franko National University of Lviv (Ukraine)}
\email{ogutik@gmail.com}
\begin{document}

\begin{abstract} For a Hausdorff topologized semilattice $X$ its {\em Lawson number $\Law(X)$} is the smallest cardinal $\kappa$ such that for any distinct points $x,y\in X$ there exists a family $\mathcal U$ of closed neighborhoods of $x$ in $X$ such that $|\U|\le\kappa$ and $\bigcap\U$ is a subsemilattice of $X$ that does not contain $y$. It follows that $\Law(X)\le\bar\psi(X)$, where $\bar\psi(X)$ is the smallest cardinal $\kappa$ such that for any point $x\in X$ there exists a family $\mathcal U$ of closed neighborhoods of $x$ in $X$ such that $|\U|\le\kappa$ and $\bigcap\U=\{x\}$.

We prove that a compact Hausdorff semitopological semilattice $X$ is Lawson (i.e., has a base of the topology consisting of subsemilattices) if and only if $\Law(X)=1$. Each Hausdorff topological semilattice $X$ has Lawson number  $\Law(X)\le\w$. On the other hand, for any  infinite cardinal $\lambda$ we construct a Hausdorff zero-dimensional semitopological semilattice $X$ such that $|X|=\lambda$ and $\Law(X)=\bar\psi(X)=\cf(\lambda)$.

A topologized semilattice $X$ is called (i) {\em $\w$-Lawson} if $\Law(X)\le\w$; (ii) {\em complete} if each non-empty chain $C\subset X$ has $\inf C\in\overline{C}$ and $\sup C\in\overline{C}$. We prove that for any complete subsemilattice $X$ of an $\w$-Lawson semitopological semilattice $Y$, the partial order $\le_X=\{(x,y)\in X\times X:xy=x\}$ of $X$ is closed in $Y\times Y$ and hence $X$ is closed in $Y$. This implies that for any continuous homomorphism $h:X\to Y$ from a compete topologized semilattice  $X$ to an $\w$-Lawson semitopological semilattice $Y$ the image $h(X)$ is closed in $Y$.
\end{abstract}
\maketitle

\section*{Introduction}

In this paper we introduce a new cardinal invariant $\Law(X)$ of a Hausdorff topologized semilattice $X$, called the Lawson number of $X$. Introducing of the Lawson number was motivated by studying the closedness properties of complete topologized semilattices.
Complete topologized semilattices were studied by the first two authors in \cite{BBm}, \cite{BBw},  \cite{BBc}, \cite{BBseq}, \cite{BBo}, \cite{BBR}. It turns out that complete semitopological semilattices share many common properties with compact topological semilattices, in particular their continuous homomorphic images in Hausdorff topological semilattices are closed.

A {\em semilattice} is any commutative semigroup of idempotents (an element $x$ of a semigroup is called an {\em idempotent} if $xx=x$). 

Each semilattice $X$ carries the natural partial order $\le_X$ defined by $x\le y$ iff $xy=x=yx$. Many properties of semilattices are defined in the language of the natural partial order. In particular, for a point $x\in X$ we can consider its upper and lower sets $${\uparrow}x:=\{y\in X:xy=x\}\mbox{ \ and \ }{\downarrow}x:=\{y\in X:xy=y\}$$ in the partially ordered set $(X,\le_X)$.


A subset $C$ of a semilattice $X$ is called a {\em chain} if $xy\in\{x,y\}$ for any $x,y\in C$. A semilattice $X$ is called {\em chain-finite} if each chain in $X$ is finite.
A semilattice is called {\em linear} if it is a chain in itself.

A semilattice endowed with a topology is called a {\em topologized semilattice}. 
A topologized semilattice $X$ is called a ({\em semi\,}){\em topological semilattice} if the semigroup operation $X\times  X\to X$, $(x,y)\mapsto xy$, is (separately) continuous. 

In \cite{Stepp1975} Stepp proved that for any homomorphism $h:X\to Y$ from a chain-finite semilattice to a Hausdorff topological semilattice $Y$, the image $h(X)$ is closed in $Y$. In \cite{BBm}, the first authors improved this result of Stepp proving the following theorem.

\begin{theorem}[Banakh, Bardyla]\label{t:cf}  For any homomorphism $h:X\to Y$ from a chain-finite semilattice to a  Hausdorff semitopological semilattice $Y$, the image $h(X)$ is closed in $Y$.
\end{theorem} 

Topological generalizations of the notion of chain-finiteness are the notions of  chain-compactness and completeness, discussed in \cite{BBo}.

A topologized semilattice $X$ is called
\begin{itemize}
\item {\em chain-compact} if each closed chain in $X$ is compact;
\item {\em complete} if each non-empty  chain $C\subset X$ has $\inf C\in\overline{ C}$ and $\sup C\in\overline{C}$.
\end{itemize}
Here $\overline{C}$ stands for the closure of $C$ in $X$.  Chain-compact and complete topologized semilattices appeared to be very helpful in studying the closedness properties of topologized semilattices, see \cite{BBm}, \cite{BBw}, \cite{BBc},  \cite{BBseq}, \cite{BBo}, \cite{BBR}, \cite{GutikRepovs2008}. By Theorem 3.1 \cite{BBm}, a Hausdorff semitopological semilattice is chain-compact if and only if it is complete (see also Theorem 4.3 \cite{BBo} for generalization of this characterization to topologized posets).
In \cite{BBm} the first two authors proved the following closedness property of complete topologized semilattices.

\begin{theorem}[Banakh, Bardyla]\label{t:kc} For any continuous homomorphism $h:X\to Y$ from a complete topologized semilattice $X$ to a Hausdorff topological semilattice $Y$, the image $h(X)$ is closed in $Y$.
\end{theorem}

Theorems~\ref{t:cf} and \ref{t:kc} motivate the following (still) open problem.

\begin{problem}\label{prob:main} Assume that $h:X\to Y$ is a continuous homomorphism from a complete topologized semilattice $X$ to a Hausdorff semitopological semilattice $Y$. Is $h(X)$ closed in $Y$?
\end{problem}

In \cite{BBseq} the first two authors gave the following partial answer to Problem~\ref{prob:main}.

\begin{theorem}[Banakh, Bardyla] For any continuous homomorphism $h:X\to Y$ from a complete topologized semilattice $X$ to a sequential Hausdorff semitopological semilattice $Y$, the image $h(X)$ is closed in $Y$.
\end{theorem}

Another partial result to Problem~\ref{prob:main} was given in \cite{BBR}.

\begin{theorem}[Banakh, Bardyla, Ravsky]\label{BBR1} For any continuous homomorphism $h:X\to Y$ from a complete topologized semilattice $X$ to a functionally Hausdorff semitopological semilattice $Y$, the image $h(X)$ is closed in $Y$.
\end{theorem}

Let us recall that a topological space $X$ is {\em functionally Hausdorff} if for any distinct points $x,y\in X$ there exists a continuous map $f:X\to\mathbb R$ such that $f(x)\ne f(y)$.

In fact, in \cite{BBR} Theorem~\ref{BBR1} was derived from the following closedness property of the partial order of a complete subsemilattice of a functionally Hausdorff semitopological semilattice.

\begin{theorem}[Banakh, Bardyla, Ravsky]\label{BBR2} For any complete subsemilattice $X$ of a functionally Hausdorff semitopological semilattice $Y$, the partial order $\le_X$ of $X$ is a closed subset of $Y\times Y$.
\end{theorem}
\smallskip
 
In this paper we shall show that the answer to Problem~\ref{prob:main} is affirmative under the additional condition that the semitopological semilattice $Y$ is $\w$-Lawson. We shall also prove a counterpart of Theorem~\ref{BBR2} for complete subsemilattices of $\w$-Lawson semitopological semilattices.

We define a topologized semilattice $X$ to be {\em $\w$-Lawson} if for any distinct points $x,y\in X$ there exists a countable family $\U$ of closed neighborhoods of $x$ such that $\bigcap\U$ is a subsemilattice of $X$ that does not contain $y$. A topologized semilattice $X$ is $\w$-Lawson if and only if it is Hausdorff and has at most countable Lawson number $\Law(X)$.

The {\em Lawson number} $\Law(X)$ of a Hausdorff topologized semilattice  $X$ is defined as the smallest cardinal $\kappa$ such that for any distinct points $x,y\in X$ there exists a family $\U$ of closed neighborhoods of $x$ such that $|\U|\le\kappa$ and $\bigcap\U$ is a subsemilattice of $X$ that does not contain $y$.

The Lawson number will be studied in more details  in Section~\ref{s:Gdelta}.
In that section we shall prove that every Hausdorff topological semilattice $X$ has $\Law(X)\le\w$. On the other hand, for any infinite cardinal $\lambda$ we shall construct a Hausdorff zero-dimensional semitopological semilattice $X$ of cardinality $|X|=\lambda$ and Lawson number $\Law(X)=\cf(\lambda)$. In Section~\ref{s:main} we prove the main result of this paper:

\begin{theorem}\label{t:main} For any complete subsemilattice $X$ of a $\w$-Lawson semitopological semilattice $Y$, the natural partial order $\le_X$ of $X$ is a closed subset of $Y\times Y$. Consequently, $X$ is closed in $Y$. 
\end{theorem}

\begin{remark} By \cite{BBR2} (and \cite{BBR3}), there exists a metrizable (Lawson)  semitopological semilattice $X$ whose partial order $\le_X$ is not closed in $X\times X$.
\end{remark}

Since the completeness is preserved by continuous homomorphisms into Hausdorff semitopological semilattices (see Lemma~\ref{l:hom}), Theorem~\ref{t:main} implies the following corollary giving a partial answer to Problem~\ref{prob:main}.

\begin{corollary}\label{c:main} For any continuous homomorphism $h:X\to Y$ from a complete topologized semilattice to an $\w$-Lawson semitopological semilattice $Y$ the image $h(X)$ is closed in $Y$.
\end{corollary}

\begin{problem} Let $X$ be a complete subsemilattice of an $\w_1$-Lawson semitopological semilattice $Y$. Is $X$ closed in $Y$? Is the natural partial order $\le_X$ of $X$ closed in $X\times X$?
\end{problem}

\section{The Lawson number of a Hausdorff topologized semilattice}\label{s:Gdelta}

It is easy to see that each $\w$-Lawson topologized semilattice is Hausdorff. In fact, $\w$-Lawson topologized semilattices can be equivalently defined as Hausdorff topologized semilattices with countable Lawson number.

\begin{definition} The {\em Lawson number} $\bar\Lambda(X)$ of a Hausdorff topologized semilattice $X$ is the smallest cardinal $\kappa$ such that for any distinct points $x,y\in X$ there exists a family $\U$ of closed neighborhoods of $x$ such that $|\U|\le\kappa$ and the intersection $\bigcap\U$ is a subsemilattice of $X$ that does not contain $y$.
\end{definition}

For any Hausdorff topologized semilattice $X$ the Lawson number $\Law(X)$ is well-defined and does not exceed the {\em closed pseudocharacter} $\bar\psi(X)$ of $X$, defined as the smallest cardinal $\kappa$ such that for any point $x\in X$ there exists a family $\U$ of closed neighborhoods of $x$ such that $|\U|\le\kappa$ and  $\bigcap\U=\{x\}$. Therefore, $\Law(X)\le\bar\psi(X)$ for any Hausdorff topologized semilattice $X$.

Observe that a topologized semilattice $X$ is $\w$-Lawson if and only if $X$ is Hausdorff and has $\Law(X)\le\w$. 

\begin{definition} A topologized semilattice $X$ is defined to be {\em $\kappa$-Lawson} for a cardinal $\kappa$ if $X$ is Hausdorff and $\Law(X)\le\kappa$. 
\end{definition}

The Lawson number admits the following simple characterization.

\begin{proposition} The Lawson number of a Hausdorff topologized semilattice $X$ is equal to the smallest cardinal $\kappa$ such that for any distinct points $x,y\in X$ there exist a closed subsemilattice $L$ in $X$ and a family $\V$ of  closed neighborhoods of $x$ such that $|\V|\le \kappa$ and $\bigcap\V\subset L\subset X\setminus\{y\}$.
\end{proposition}

\begin{proof} We should prove that $\Law(X)=\Lambda(X)$ where $\Lambda(X)$ is  the smallest cardinal $\kappa$ such that for any distinct points $x,y\in X$ there exist a closed subsemilattice $L$ in $X$ and a family $\V$ of  closed neighborhoods of $x$ such that $|\V|\le \kappa$ and $\bigcap\V\subset L\subset X\setminus\{y\}$.
\smallskip

To see that $\Lambda(X)\le\Law(X)$, for any distinct points $x,y\in X$ use the definition of $\Law(X)$ and find a family $\U$ of closed neighborhoods of $x$ such that $|\U|\le\Law(X)$ and $L:=\bigcap\U$ is a closed subsemilattice of $X$ that does not contain $y$. Then $\bigcap\U=L\subset X\setminus\{y\}$, witnessing that $\Lambda(X)\le\bar \Lambda(X)$.
\smallskip 

To see that $\Law(X)\le\Lambda(X)$, for any distinct points $x,y\in X$ use the definition of $\Lambda(X)$ and find a closed subsemilattice $L$ of $X$ and  a family $\V$ of closed neighborhoods of $x$ such that $|\V|\le\Lambda(X)$ and $\bigcap\V\subset L\subset X\setminus\{y\}$. Then $\U:=\{V\cup L:V\in\V\}$ is a family of closed neighborhoods of $x$ such that $|\U|\le |\V|\le\Lambda(X)$ and $\bigcap\U=(\bigcap\V)\cup L=L$ is a closed subsemilattice of $X$ that does not contain $y$ and witnesses that $\Law(X)\le\bar\Lambda(X)$.
\end{proof}

The notion of a $1$-Lawson semilattice extends the well-known notion of a  Lawson semilattice (or else a topologized semilattice with small subsemilattices), introduced and studied by Lawson \cite{Lawson} (see also \cite[Chapter 2]{CHK}). Following \cite[p.12]{CHK}, we define a topologized semilattice $X$ to be 
\begin{itemize}
\item {\em Lawson} if it has a base of the topology consisting of open subsemilattices;
\item a {\em $V$-semilattice} if for any point $x\in X$ and $y\notin {\uparrow}x$ there exists a point $v\in X\setminus{\downarrow} y$ such that ${\uparrow}v$ is a neighborhood of $x$ in $X$;
\item {\em $\II$-separated} if for any distinct points $x,y\in X$ there exists a continuous homomorphism $f:X\to\II$ such that $f(x)\ne f(y)$.
\end{itemize} Here by $\II$ we denote the unit interval $[0,1]$ endowed with the semilattice operation $\min$.

\begin{proposition} For a Hausdorff semitopological semilattice $X$, consider the following conditions:
\begin{enumerate}
\item $X$ is $1$-Lawson;
\item $X$ is Lawson;
\item $X$ is a $V$-semilattice;
\item $X$ is $\II$-separated.
\end{enumerate}
Any of the conditions \textup{(2),(3),(4)} implies \textup{(1)}. If the space $X$ is compact, then the conditions \textup{(1)--(4)} are equivalent.
\end{proposition}

\begin{proof} To prove that $X$ is $1$-Lawson, fix any distinct points $x,y\in X$. We need to find a closed subsemilattice $L$ in $X$ that contains $x$ in its interior but does not contain $y$.
\smallskip

$(2)\Ra(1)$ Assume that $X$ is Lawson. Since $X$ is Hausdorff, there exists a closed neighborhood $N_x\subset X$ of $x$ such that $y\notin N_x$. Since $X$ is Lawson, there exists an open subsemilattice $V\subset X$ such that $x\in V\subset N_x$. Since $X$ is a semitopological semilattice, the closure $\overline{V}$ of the semilattice $V$ is a closed subsemilattice that contains $x$ in its interior but does not contain $y$. Therefore, the semilattice $X$ is $1$-Lawson.
\smallskip

$(3)\Ra(1)$ Assume that $X$ is a $V$-semilattice. If $x\notin{\downarrow}y$, then $y\notin{\uparrow}x$ and there exists an element $v\in X\setminus{\downarrow}y$ such that the upper set ${\uparrow}v$ contains $x$ in its interior. Since $X$ is a Hausdorff semitopological semilattice the upper set ${\uparrow}v=\{z\in X:zv=v\}$ is closed. Then the closed subsemilattice ${\uparrow}v$ is a neighborhood of $x$ that does not contain $y$. If $x\in{\downarrow}y$, then $x\notin{\uparrow}y$. Since $X$ is a $V$-semilattice, there exists an element $u\in X\setminus{\downarrow}x$ such that the upper set ${\uparrow}u$ contains $y$ in its interior. Observe that the complement $U:=X\setminus{\uparrow}u$ is an open subsemilattice of $X$, containing $x$. Then $\overline{U}$ is a closed subsemilattice of $X$, which is a neighborhood of $x$ that does not contain $y$.
\smallskip

$(4)\Ra(1)$ Assuming that $X$ is $\II$-separated, we can find a continuous homomorphism $f:X\to\II$ such that $f(x)\ne f(y)$. Choose any closed neighborhood $N\subset\II$ of $f(x)$ such that $f(y)\notin N$. Then $f^{-1}(N)$ is a closed subsemilattice in $X$ that contains $x$ in its interior but does not contain $y$.
\smallskip

Now assume that $X$ is compact. In this case the conditions (1)--(4) are equivalent by Theorem 7.1 in \cite{BBw}. In fact, the equivalence of the conditions (2) and (4) is a classical result of Lawson \cite{Lawson} and \cite{Law74}.
\end{proof}

\begin{example} The topological semilattice $\IZ^\omega$ with the Tychonoff product topology and coordinatewise operation of minimum is Lawson and $\II$-separated but not a $V$-semilattice.
\end{example}

 By \cite[Example~2.21]{CHK}, there exists a metrizable compact topological
semilattice, which is not Lawson and hence is not $1$-Lawson. However, such a semilattice necessarily is $\w$-Lawson as shown by the following simple proposition.

\begin{proposition}\label{p:Gdelta-Lawson} Each Hausdorff topological semilattice $X$ is $\w$-Lawson.
\end{proposition}

\begin{proof} Given two distinct points $x,y\in X$, choose a decreasing sequence $(U_n)_{n\in\w}$ of open neighborhoods of $x$ such that $y\notin\overline{U_0}$ and $U_{n}\cdot U_n\subset U_{n-1}$ and hence $\overline{U_n}\cdot\overline{U_n}\subset \overline{U_{n-1}}$ for all $n\in\IN$. The choice of $U_0$ is possible by the Hausdorff property of $X$, and the choice of the neighborhoods $U_n$ is possible by the continuity of the semilattice operation at $(x,x)$. 
 It follows that the intersection $\bigcap_{n\in\w}\overline{U_n}$ is a closed subsemilattice of $X$ containing $x$ but not $y$. 
\end{proof}

\begin{corollary} Each compact Hausdorff semitopological semilattice is $\w$-Lawson.
\end{corollary}

\begin{proof} By \cite{Law74}, each compact Hausdorff semitopological semilattice is a topological semilattice and by Proposition~\ref{p:Gdelta-Lawson}, is a $\w$-Lawson.
\end{proof}

Let us also notice the following trivial (but useful) fact.

\begin{proposition} Each \textup{(}Hausdorff\textup{)} linear topologized semilattice $X$ is Lawson \textup{(}and  $1$-Lawson\textup{)}.
\end{proposition}




Let us recall that a topological space $X$ is {\em Urysohn} if any distinct points in $X$ have disjoint closed neighborhoods. It is clear that each Urysohn space is Hausdorff.

We define a topological space $X$ to be {\em $\kappa$-Urysohn} for a cardinal $\kappa$ if for any distinct points $x,y\in X$ there are families $\U_x$ and $\U_y$ of open sets on $X$ such that $\max\{|\U_x|,|\U_y|\}\le\kappa$, $x\in\bigcap\U_x$, $y\in\bigcap\U_y$ and the sets $\bigcap_{U\in\U_x}\overline{U}$ and $\bigcap_{V\in\U_y}\overline{V}$ are disjoint.

 It is easy to see that a topological space is Urysohn if and only if it is $1$-Urysohn.
 
\begin{example} For every infinite cardinal $\kappa$, there exists a Hausdorff Lawson topological semilattice, which is not $\kappa$-Urysohn.
\end{example}

\begin{proof} Take any ordinal $\lambda$ of cofinality $\mathrm{cf}(\lambda)>\kappa$ (for example, put $\lambda:=\kappa^+$). Consider the set $L=\{x_\alpha\}_{\alpha\le\lambda}\cup\{z\}\cup\{y_\alpha\}_{\alpha\le \lambda}$ of pairwise distinct points endowed with the linear order  in which $x_\alpha<x_\beta<z<y_\beta<y_\alpha$ for any  ordinals $\alpha<\beta\le\lambda$. Let $\ddot L:=L\setminus\{x_{\lambda},y_{\lambda}\}$. On the set $$X=(\ddot L\times[0,\lambda))\cup(\{x_{\lambda},y_{\lambda}\}\times\{\lambda\})$$ consider the semilattice operation
$$(x,\alpha)\cdot(y,\beta):=\begin{cases}
(\min\{x,y\},\min\{\alpha,\beta\})&\mbox{if $\alpha,\beta<\lambda$};\\
(\min\{x,z\},\alpha),&\mbox{if $\alpha<\lambda=\beta$};\\
(\min\{z,y\},\beta),&\mbox{if $\beta<\lambda=\alpha$};\\
(\min\{x,y\},\lambda),&\mbox{if $\alpha=\lambda=\beta$}.
\end{cases}
$$
Endow $X$ with the topology $\tau$ consisting of all sets $U\subset X$ satisfying the following three conditions:
\begin{itemize}
\item if $(z,\alpha)\in U$ for some $\alpha\in[0,\lambda)$, then $\{(x_\gamma,\alpha),(y_\gamma,\alpha):\beta<\gamma<\lambda\}\subset U$ for some $\beta\in[0,\lambda)$;
\item if $(x_{\lambda},\lambda)\in U$, then $\{(x_\beta,\gamma):\beta,\gamma\in[\alpha,\lambda)\}\subset U$ for some $\alpha\in[0,\lambda)$;
\item if $(y_{\lambda},\lambda)\in U$, then $\{(y_\beta,\gamma):\beta,\gamma\in[\alpha,\lambda)\}\subset U$ for some $\alpha\in[0,\lambda)$.
\end{itemize}
Taking into account that $\mathrm{cf}(\lambda)>\kappa$, we can show that $(X,\tau)$ is a required Hausdorff Lawson topological semilattice which is not $\kappa$-Urysohn.
\end{proof}  

Now we construct Hausdorff zero-dimensional semitopological semilattices having an arbitrarily large Lawson number. We recall that a topological space is {\em zero-dimensional} if it has a base of the topology consisting of open-and-closed sets.

\begin{example}\label{e:main} For any infinite cardinal $\lambda$ there exists a Hausdorff zero-dimensional semitopological semilattice $X$ such that $|X|=\lambda$ and $\Law(X)=\bar\psi(X)=\mathrm{cf}(\lambda)$. 
\end{example}

\begin{proof} Consider the set  
$$X:=\{A\subset \lambda:\mbox{$A=\lambda$ or $A$ is finite}\}$$endowed with the semilattice operation of union. This semilattice has cardinality $|X|=\lambda$. 
Here we identify the cardinal $\lambda$ with the set $[0,\lambda)$ of all ordinals smaller than $\lambda$.

Now the trick is to introduce an appropriate topology on the semilattice $X$. For this we define several kinds of sets in $\lambda$.

A finite subset $A\subset \lambda$ is defined to be {\em sparse} if $|A\cap[\alpha,\alpha+\w)|\le 1$ for any ordinal $\alpha\in\lambda$.

For a set $A$ and an ordinal $\alpha\in\lambda$ consider the set
$$S[A;\alpha]:=\{B\in X:\mbox{$B\cap[0,\alpha)=A\cap[0,\alpha)$ and $B\cap[\alpha,\lambda)$ is sparse}\},$$and observe that $\lambda\notin S[A;\alpha]$.

Let $\alpha\in\lambda$ be an ordinal, $n$ be a finite ordinal and $\e$ be a positive real number. A subset $A\subset\lambda$ is called {\em $(\alpha,n,\e)$-fat} if there exists a limit ordinal $\beta\in[\alpha,\lambda)$ and a finite ordinal $m>n$ that
\begin{itemize}
\item[(i)] the set $A\cap[\beta+\w,\lambda)$ is sparse;
\item[(ii)] $[\beta,\beta+\w)\cap A=[\beta,\beta+m]$;
\item[(iii)] the set $[0,\beta)\cap A$ is finite and has cardinality $<\e\cdot m$.
\end{itemize}
The conditions (i),(ii) ensure that the ordinal $\beta$ is unique.

Consider the subset 
$$F[\alpha,n,\e]:=\{\lambda\}\cup\{A\in X:\mbox{$A$ is $(\alpha,n,\e)$-fat}\}$$of $X$.

Now we define a topology $\tau$ on the semilattice $X$. This topology consists of the sets $U\subset X$ satisfying the following two conditions:
\begin{itemize}
\item[(a)] for any finite subset $A\in U$ of $\lambda$ there exists an ordinal $\alpha\in\lambda$ such that $A\in S[A,\alpha]\subset U$;
\item[(b)] if $\lambda\in U$, then there exist ordinals $\alpha\in\lambda$, $k\in\w$, and a positive real number $\e$ such that $F[\alpha,k,\e]\subset U$.
\end{itemize}
It is easy to see that $\tau$ is a well-defined topology on $X$. Now we show that this topology is Hausdorff and zero-dimensional.

\begin{claim}\label{cl1} For an element $A\in X$ and any ordinal $\alpha\in\lambda$ the set $S[A;\alpha]$ is open in $(X,\tau)$.
\end{claim}

\begin{proof} To see that $S[A;\alpha]$ is open, take any element $B\in S[A;\alpha]$ and observe that $B\cap[0,\alpha)=A\cap[0,\alpha)$ and $B\cap[\alpha,\lambda)$ is sparse. Then $B\in S[B;\alpha]\subset S[A;\alpha]\subset X\setminus\{\lambda\}$ and the set $S[A;\alpha]$ is open by the definition of the topology $\tau$.
\end{proof}

\begin{claim}\label{cl2} For any ordinals $\alpha\in\lambda$, $n\in\w$ and a positive real number $\e$, the set
$F[\alpha,n,\e]$ is open and closed in $(X,\tau)$.
\end{claim}

\begin{proof} Given any element $A\in F[\alpha,n,\e]\setminus\{\lambda\}$, find a unique limit ordinal $\beta\ge\alpha$ witnessing that $A$ is $(\alpha,n,\e)$-fat. Then $S[A;\beta+\w]\subset F[\alpha,n,\e]$, witnessing that the set $F[\alpha,n,\e]$ is $\tau$-open.

To see that this set is closed in $(X,\tau)$, choose any set $A\in X\setminus F[\alpha,n,\e]$. It follows that $A$ is a finite subset of $\lambda$, which is not $(\alpha,n,\e)$-fat. Let $\beta$ is the smallest limit ordinal such that the intersection $A\cap[\beta+\w,\lambda)$ is sparse. Let $m\in\w$ be the smallest finite ordinal such that $A\cap[\beta,\beta+\w)\subset [\beta,\beta+m]$. Since $A$ is not $(\alpha,n,\e)$-fat, one of the following conditions holds:
\begin{enumerate}
\item $\beta<\alpha$;
\item $m\le n$;
\item $\beta\ge\alpha$ and $m>n$ but $[\beta,\beta+m]\not\subset A$;
\item $\beta\ge\alpha$, $m>n$, $[\beta,\beta+m]\subset A$ but $|A\cap[0,\beta)|\ge \e\cdot m$.
\end{enumerate}
In all these cases $S[A;\beta+\w]$ is a $\tau$-open  neighborhood of $A$ such that $S[A;\beta+\w]\cap F[\alpha,n,\e]=\emptyset$.
\end{proof}
  
\begin{claim}\label{cl3}  For any finite set $A\in X$ of $\lambda$ and any ordinal $\alpha\in\lambda$, the set $S[A;\alpha]$ is closed in $(X,\tau)$.
\end{claim}

\begin{proof} Take any element $B\in X\setminus S[A;\alpha]$. If $B=\lambda$, then $F[\alpha,1,1]$ is a neighborhood of $B$, disjoint with $S[A;\alpha]$.

If $B\ne \lambda$, then $B\notin S[A;\alpha]$ implies that either $B\cap[\alpha,\lambda)$ is not sparse or $B\cap[\alpha,\lambda)$ is sparse but $B\cap[0,\alpha)\ne A\cap[0,\alpha)$. In the latter case $S[B;\alpha]$ is a $\tau$-open neighborhood of $B$, disjoint with $S[A;\alpha]$. So, we assume that $B\cap[\alpha,\lambda)$ is not sparse. In this case we can choose any ordinal $\beta\in\lambda$ with $B\subset[0,\beta)$ and observe that $S[B,\beta]$ is a $\tau$-open neighborhood of $B$ such that $S[B;\beta]\cap S[A;\alpha]=\emptyset$.
\end{proof}

\begin{claim}\label{cl:H} The topology $\tau$ is Hausdorff. 
\end{claim}

\begin{proof} Take any distinct elements $A,B\in X$. If $A$ and $B$ are finite subsets of $X$, then we can find an ordinal $\alpha\in\lambda$ such that $A\cup B\subset [0,\alpha)$ and observe that $S[A;\alpha]$ and $S[B;\alpha]$ are disjoint $\tau$-open neighborhoods of the elements $A$ and $B$, respectively.

If $B=\lambda$, then $A$ is a finite set, contained in $[0,\alpha)$ for some ordinal  $\alpha\in\lambda$.  In this case $S[A;\alpha]$ and $F[\alpha,1,1]$ are disjoint $\tau$-open neighborhoods of $A$ and $B$, respectively,

The case $A=\lambda$ can be considered by analogy.
\end{proof}

Claims~\ref{cl1}--\ref{cl:H} show that the topology $\tau$ is Hausdorff and zero-dimensional.

\begin{claim} The topologized semilattice $(X,\tau)$ is semitopological.
\end{claim}

\begin{proof} Given any element $a\in X$, we should prove that the shift $s_a:X\to X$, $s_a:x\mapsto ax$, is continuous. If $a=\lambda$, then $s_a(X)=\{\lambda\}$ is a singleton, so the continuity of $s_a$ is trivial. So, we assume that $a$ is a finite subset of $\lambda$. To check the continuity of the shift $s_a$ at a point $x\in X$, fix any neighborhood $O_{ax}\in\tau$ of the point $ax=x\cup a$. 

If $x\ne\lambda$, then $ax\ne\lambda$ and by the definition of the topology $\tau$, there exists an ordinal $\alpha\in\lambda$ such that $ax\in S[ax;\alpha]\subset O_{ax}$. Replacing $\alpha$ by a larger ordinal, we can assume that $ax\subset [0,\alpha)$. Then $O_x:=S[x;\alpha]$ is a $\tau$-open neighborhood such that $s_a(O_x)\subset S[ax;\alpha]\subset O_{ax}$.

If $x=\lambda$, then $ax=\lambda$ and by the definition of the topology $\tau$, there exist $\alpha\in\lambda$, $k\in\w$ and $\e>0$ such that $F[\alpha,k,\e]\subset O_{ax}$. Replacing $\alpha$ by a larger ordinal, we can assume that $a\subset[0,\alpha)$.   Replacing $k$ by a larger number, we can assume that $|a|\le \frac12\e k$. In this case $O_x:=F[\alpha,k,\frac12\e]$ is a $\tau$-open neighborhood of $x=\lambda$ such that $s_a(O_x)\subset F[\alpha,k,\e]\subset O_{ax}$.
\end{proof} 


\begin{claim}\label{cl6} Let $\U\subset\tau$ be a family of open sets and $L$ be a $\tau$-closed subsemilattice in $X$ such that $|\U|<\cf(\lambda)$ and $\emptyset\ne\bigcap\U\subset L$. Then $\lambda\in L$.
\end{claim}

\begin{proof} Fix any element $x\in\bigcap\U$. If $x=\lambda$, then $\lambda=x\in\bigcap\U\subset L$ and we are done. So, we assume that $x$ is a finite subset of $\lambda$. To show that $\lambda\in L$, take any neighborhood $O_\lambda\in\tau$ of $\lambda\in X$. By Claim~\ref{cl2}, there exist ordinals $\alpha\in\lambda$, $k\in\w$ and a positive real number $\e$ such that $F[\alpha,k,\e]\subset O_\kappa$.  

By Claim~\ref{cl1} and $|\U|\le\kappa<\mathrm{cf}(\lambda)$, there exists a limit ordinal $\beta\in[\alpha,\lambda)$ such that $x\in[0,\beta)$ and $x\in S[x;\beta]\subset\bigcap\U\subset L$.  Choose a finite ordinal $n> k$ such that $|x|<\e n$.
Observe that for every ordinal $\gamma\in [\beta,\beta+n]$ the set $x\cup\{\gamma\}$ belongs to the semilattice $L\supset S[x;\beta]$. Since $L$ is a subsemilattice, the $(\beta,k,\e)$-fat set $x\cup[\beta,\beta+n]$ belongs to $L\cap F[\beta,k,\e]\subset L\cap O_\lambda$.
\end{proof} 

\begin{claim} The semitopological semilattice $(X,\tau)$ has $\Law(X,\tau)=\bar\psi(X,\tau)=\mathrm{cf}(\lambda)$.
\end{claim}

\begin{proof} Claim~\ref{cl6} implies that $\cf(\lambda)\le\Law(X,\tau)$. Since $\Law(X,\tau)\le\bar\psi(X,\tau)$, it remains to prove that $\bar\psi(X,\tau)\le\cf(\lambda)$.

Choose a cofinal subset $C\subset\lambda$ of cardinality $|C|=\cf(\lambda)$.
To see that $\bar\psi(X,\tau)\le\cf(\lambda)$, take any $A\in X$. If $A$ is a finite subset of $\lambda$, then $A\subset [0,\alpha)$ for some ordinal $\alpha\in\lambda$. Then $\{A\}=\bigcap_{\alpha\le \gamma\in C}S[A;\gamma]$. If $A=\lambda$, then $\{A\}=\{\lambda\}=\bigcap_{\gamma\in C}F[\gamma;1,1]$. In both cases the singleton $\{A\}$ is the intersection of $\cf(\lambda)$ many closed neighborhoods of $A$, witnessing that $\bar\psi(X,\tau)\le\cf(\lambda)$.
\end{proof}
\end{proof}

\section{Complete topologized semilattices}

In this section we recall some known properties and characterizations of complete topologized semilattices. 

By a {\em poset} we understand a set endowed with a partial order. A {\em topologized poset} is a poset endowed with a topology. So, each topologized semilattice is a topologized poset.

A subset $D$ of a poset $(X,\le)$ is called 
\begin{itemize}
\item a {\em chain} if any elements $x,y\in D$ are comparable in the sense that $x\le y$ or $y\le x$;
\item {\em up-directed} if for any $x,y\in D$ there exists $z\in D$ such that $x\le z$ and $y\le z$;
\item {\em down-directed} if for any $x,y\in D$ there exists $z\in D$ such that $z\le x$ and $z\le y$.
\end{itemize}
It is clear that each chain in a poset is both up-directed and down-directed.

A topologized posed $X$ is defined to be 
\begin{itemize}
\item {\em up-complete} if any nonempty up-directed subset $U\subset X$ has the smallest upper bound $\sup U\in\overline{U}$ in $X$;
\item {\em down-complete} if any nonempty down-directed subset $D\subset X$ has the greatest lower bound $\inf D\in\overline{D}$ in $X$.
\end{itemize}

The proof of the following classical characterization can be found in \cite{Iwamura}, \cite{Bruns}, \cite{Mark} or \cite[2.2]{BBo}.

\begin{proposition}\label{Iwamura} For a topologized poset $X$ the following conditions are equivalent:
\begin{enumerate}
\item $X$ is up-complete;
\item each non-empty chain $C\subset X$ has the smallest upper bound $\sup C\in\overline{C}$ in $X$.
\end{enumerate}
\end{proposition}

Here for a subset $A$ of a topological space $X$ by $\overline{A}$ we denote the closure of $A$ in $X$.

Proposition~\ref{Iwamura} implies the following useful characterization of completeness in topologized semilattices.

\begin{corollary}\label{c:chara} A topologized semilattice $X$ is complete if and only if it is up-complete and down-complete.
\end{corollary}

This corollary implies that each closed subsemilattice of a complete topologized semilattice has the smallest element. 

A topologized semilattice $Y$ is called {\em $\uparrow$-closed} if for every $y\in Y$ the upper set ${\uparrow}y=\{x\in Y:xy=y\}$ is closed in $Y$. It is easy to see that each $T_1$ semitopological semilattice is ${\uparrow}$-closed.

The following lemma (that can be derived from Corollary~\ref{c:chara}) is proved in \cite[Lemma~5.3]{BBo}.

\begin{lemma}\label{l:hom} Let $h:X\to Y$ be a continuous surjective homomorphism between topologized semilattices. If $X$ is complete and $Y$ is $\uparrow$-closed, then the topologized semilattice $Y$ is complete.
\end{lemma}

\section{Proof of Theorem~\ref{t:main} and Corollary~\ref{c:main}}\label{s:main}

The  proof of Theorem~\ref{t:main} is based on the following lemma.

\begin{lemma}\label{l1} Let $X$ be a complete subsemilattice of a semitopological semilattice $Y$. Let a pair $(x,y)\in Y\times Y$ belong to the closure of the natural partial order $\le_X$ of $X$ in $Y\times Y$, and let $\{U_n\}_{n\in\w}$, $\{V_n\}_{n\in\w}$ be sequences of closed neighborhoods of the points $x$ and $y$ in $Y$, respectively. Then there exist points $x'\in X\cap\bigcap_{n\in\w}U_n$ and $y'\in X\cap\bigcap_{n\in\w}V_n$ such that $x'\le y'$.
\end{lemma}

\begin{proof} Replacing each set $U_n$ by $\bigcap_{i\le n}U_i$, we can assume that $U_{n+1}\subset U_n$ for all $n\in\w$. By the same reason, we can assume that the sequence $(V_n)_{n\in\w}$ is decreasing. For every $n\in\w$ denote by $U_n^\circ$ and $V_n^\circ$ the interiors of the sets $U_n$ and $V_n$ in $Y$.

By induction we shall construct sequences $(x_n)_{n\in\w}$ and $(y_n)_{n\in\w}$ of points of $X$ such that for every $n\in\w$ the following conditions are satisfied:
\begin{enumerate}
\item[$(1_n)$] $x_n\le y_n$;
\item[$(2_n)$] $\{x_i\cdots x_n,x_i\cdots x_nx\}\subset U_i^\circ$ for all $i\le n$;
\item[$(3_n)$] $\{y_i\cdots y_n,y_i\cdots y_ny\}\subset V_i^\circ$ for all $i\le n$.
\end{enumerate}

To choose the initial points $x_0,y_0$, use the separate continuity of the semilattice operation and find neighborhoods $U'_0\subset U_0^\circ$ and $V'_0\subset V_0^\circ$ of $x$ and $y$ in $Y$ such that $U'_0x\subset U_0^\circ$ and $V'_0y\subset V_0^\circ$. By our assumption, there are points $x_0\in X\cap U'_0$ and $y_0\in X\cap V'_0$ such that  $x_0\le y_0$. The choice of the neighborhoods $U'_0$ and $V'_0$ ensures that the conditions $(2_0)$ and $(3_0)$ are satisfied.

Now assume that for some $n\in\IN$ points $x_0,\dots,x_{n-1}$ and $y_0,\dots,y_{n-1}$ of $X$ are chosen so that the conditions $(1_{n-1})$--$(3_{n-1})$ are satisfied. The condition $(2_{n-1})$ implies that for every $i\le n$ we have the inclusion $x_i\cdots x_{n-1}xx=x_i\cdots x_{n-1}x\in U_i^\circ$ (if $i=n$, then we understand that $x_i\cdots x_{n-1}x=x$).
Using the continuity of the shift $s_x:Y\to Y$, $s_x:z\mapsto xz$, we can  find a neighborhood $U_{n}'\subset Y$ of $x$ such that $x_i\cdots x_{n-1}\cdot(U_n'\cup U_n'x)\subset U_i^\circ$ for every $i\le n$. By analogy, we can find a neighborhood $V_n'\subset Y$ of $y$ such that   $y_i\cdots y_{n-1}\cdot(V_n'\cup V_n'y)\subset V_i^\circ$ for every $i\le n$.
By our assumption, there are points $x_n\in X\cap U_n'$ and $y_n\in X\cap V_n'$ such that $x_n\le y_n$. The choice of the neighborhoods $U'_0$ and $V'_0$ ensures that the conditions $(2_n)$ and $(3_n)$ are satisfied.
This completes the inductive step.
\smallskip

Now for every $i\in\w$ consider the chain $C_i=\{x_i\cdots x_n:n\ge i\}\subset U_i^\circ$ in $X$.  By the completeness of $X$, this chain has $\inf C_i\in X\cap\overline{C_i}\subset X\cap\overline{U_i^\circ}\subset X\cap U_i$. Observing that $\inf C_i\le x_ix_{i+1}\cdots x_n\le x_{i+1}\cdots x_n$ for all $i>n$, we see that $\inf C_i$ is a lower bound of the chain $C_{i+1}$ and hence $\inf C_i\le\inf C_{i+1}$. By the completeness of $X$, for every $i\in\w$ the chain $D_i:=\{\inf C_j:j\ge i\}\subset U_i$ has $\sup D_i\in X\cap\overline{D_i}\subset X\cap  U_i$. Since the sequence $(\inf C_i)_{i\in\w}$ is increasing, we get $\sup D_0=\sup D_i\in X\cap U_i$ for all $i\in\w$. Consequently, $\sup D_0\in X\cap\bigcap_{i\in\w}U_i$.

By analogy, for every $k\in\w$ consider the chain $E_i=\{y_i\cdots y_n:n\ge i\}\subset V_i^\circ$ in $X$.  By the completeness of $X$, this chain has $\inf E_i\in X\cap\overline{E_i}\subset X\cap\overline{V_i^\circ}\subset X\cap V_i$. By the completeness of $X$, for every $i\in\w$ the chain $F_i:=\{\inf E_j:j\ge i\}\subset V_i$ has $\sup F_i\in X\cap\overline{F_i}\subset X\cap \overline{V_i}=X\cap V_i$. Since the sequence $(\inf E_i)_{i\in\w}$ is increasing, we get $\sup F_0=\sup F_i\in X\cap V_i$ for all $i\in\w$. Consequently, $\sup F_0\in X\cap\bigcap_{i\in\w}V_i$.

To finish the proof of Lemma~\ref{l1}, it suffices to show that $\sup D_0\le \sup F_0$. The inductive conditions $(1_n)$, $n\in\w$, imply that $\inf C_i\le \inf E_i$ for all $i\in\w$ and $\sup D_0=\sup\{\inf C_i:i\in\w\}\le \sup\{\inf E_i:i\in\w\}=\sup F_0$.
\end{proof}


The following two lemmas imply Theorem~\ref{t:main}.

\begin{lemma}\label{l2} Let $Y$ be an $\w$-Lawson semitopological semilattice. For any complete subsemilattice $X\subset Y$ the natural partial order $\le_X$ of $X$ is closed in $Y\times Y$.
\end{lemma}

\begin{proof} By Corollary~\ref{c:chara}, the complete semitopological semilattice $X$ is both up-complete and down-complete. To show that the partial order $\le_X:=\{(x,y)\in X\times X:x\le y\}$ is closed in $Y\times Y$, take any pair $(y_1,y_2)$ in the closure of the set $\le_X$ in $Y\times Y$. For every $i\in\{1,2\}$, let $\mathfrak U_i$ be the set of all countable families $\U$ of closed neighborhoods of $y_i$ in $Y$ such that $\bigcap\U$ is a subsemilattice of $Y$. By Lemma~\ref{l1}, for any  $\U_1\in\mathfrak U_1$ and $\U_2\in\mathfrak U_2$ there are points $x_1\in X\cap\bigcap\U_1$ and $x_2\in X\cap\bigcap\U_2$ such that $x_1\le x_2$. In particular, the closed subsemilattice $X\cap\bigcap\U_1$ is not empty and has the smallest element $\inf (X\cap\bigcap\U_1)\in X$ (by the down-completeness of $X$). Denote this smallest element by $x(\U_1)$. It follows that $x(\U_1):=\inf(X\cap\bigcap\U_1)\le x_1\le x_2$. Consequently, the closed subsemilattice  $({\uparrow}x(\U_1))\cap (X\cap\bigcap\U_2)\ni x_2$ is not empty and has the smallest element (by down-completeness of $X$), which will be denoted by $y(\U_1,\U_2)$. Observe that $x(\U_1)\in X\cap\bigcap\U_1$, $y(\U_1,\U_2)\in X\cap\bigcap\U_2$ and $x(\U_1)\le y(\U_1,\U_2)$. For any families $\U_1\in\mathfrak U_1$ and $\U_2,\U_2'\in \mathfrak U_2$ with $\U_2\subseteq \U_2'$ we have $y(\U_1,\U_2)\le y(\U_1,\U_2')$. Therefore, the set $\{y(\U_1,\U_2):\U_2\in\mathfrak U_2\}\subset X$ is up-directed  and by the up-completeness of $X$, it has the smallest upper bound in $X$, which will be denoted by $y(\U_1)$. It follows that $x(\U_1)\le y(\U_1)$. We claim that $y(\U_1)=y_2$. In the opposite case we can use the $\w$-Lawson property of $Y$ and choose a countable family $\U_2'\in\mathfrak U_2$ such that $y(\U_1)\notin \bigcap\U_2'$. Taking into account that the  set $\{y(\U_1,\U_2\cap\U_2'):\U_2\in\mathfrak U_2\}$ is cofinal in $\{y(\U_1,\U_2):\U_2\in\mathfrak U_2\}$, we conclude that $$y(\U_1)=\sup\{y(\U_1,\U_2):\U_2\in\mathfrak U_2\}=\sup\{y(\U_1,\U_2\cap\U_2'):\U_2\in\mathfrak U_2\}\in {\textstyle\bigcap}\U_2',$$
which contradicts the choice of the family $\U'_2$. This contradiction shows that $y_2=y(\U_1)\in X$. Now we see that $x(\U_1)\le y(\U_1)=y_2$ for every $\U_1\in\mathfrak U_1$. By the up-completeness of the semitopological semilattice $X$, the up-directed subset $\{x(\U_1):\U_1\in\mathfrak U_1\}$ has the smallest upper bound $x\in X$. It follows from $x(\U_1)\le y_2$ for all $\U_1\in\mathfrak U_1$ that $x\le y_2$.

It remains to check that $x=y_1$. In the opposite case, using the $\w$-Lawson property of $X$, we can find a countable family $\U_1'\in\mathfrak U_1$ such that $x\notin \bigcap\U_1'$. Taking into account that the family $\{x(\U_1\cap\U_1'):\U_1\in\mathfrak U_1\}$ is cofinal in $\{x(\U_1):\U_1\in\mathfrak U_1\}$, we conclude that $$x=\sup\{x(\U_1):\U_1\in\mathfrak U_1\}=\sup\{x(\U_1\cap\U_1'):\U_1\in\mathfrak U_1\}\in {\textstyle\bigcap}\U_1',$$
which contradicts the choice of $\U_1'$. This contradiction shows that $y_1=x\in X$. Therefore we obtain that $(y_1,y_2)\in X\times X$ and $y_1=x\le y_2$, which means that $(y_1,y_2)\in\;\le_X$.
\end{proof}

\begin{lemma}\label{l3} Each complete subsemilattice $X$ of  an $\w$-Lawson semitopological semilattice $Y$ is closed in $Y$.
\end{lemma}

\begin{proof} By Lemma~\ref{l2}, the partial order $\le_X:=\{(x,y)\in X\times X:xy=x\}$ is a closed subset of $Y\times Y$. By Corollary~\ref{c:chara}, the complete semilattice $X$ has the smallest element $\min X\in X$. Consider the continuous map $f:Y\to Y\times Y$, $f:y\mapsto (\min X,y)$, and observe that $X=f^{-1}(\le_X)$ is a closed subset of $X$, being the preimage of the closed set $\le_X$ under the continuous map $f$.
\end{proof}

Finally, we prove Corollary~\ref{c:main}.

\begin{lemma}\label{l4} For every continuous homomorphism $h:X\to Y$ from a complete topologized semilattice $X$ to an $\w$-Lawson semitopological semilattice $Y$, the image $h(X)$ is closed in $Y$.
\end{lemma}

\begin{proof} Observe that the $\w$-Lawson property of $Y$ implies that the semitopological semilattice $Y$ is Hausdorff and hence ${\uparrow}$-closed. By Lemma~\ref{l:hom}, the semitopological semilattice $h(X)$ of $Y$ is complete and by Lemma~\ref{l3}, $h(X)$ is closed in $Y$.
\end{proof}

\end{document}